\documentclass{amsart}
\usepackage[utf8]{inputenc}
\usepackage{slashed}
\usepackage{textcomp} 
\usepackage{amsmath,amstext,amsopn,amssymb, amsfonts, amsthm, mathtools}  
\usepackage{mathrsfs}
\usepackage{graphicx}
\usepackage{xcolor}
\usepackage{geometry} 
\usepackage{array} 
\usepackage{textcomp} 
\usepackage[all,cmtip]{xy} 
\usepackage{bbm}
\usepackage{varioref}
\usepackage{pdfsync}

\usepackage{enumitem}
\makeatletter
\def\namedlabel#1#2{\begingroup
    #2%
    \def\@currentlabel{#2}%
    \phantomsection\label{#1}\endgroup
}
\makeatother

\usepackage[pdftex,bookmarks,colorlinks,breaklinks]{hyperref}  
\definecolor{dullmagenta}{rgb}{0.4,0,0.4}   
\definecolor{darkblue}{rgb}{0,0,0.4}
\definecolor{darkgreen}{rgb}{0,0.4,0}
\hypersetup{linkcolor=darkblue,citecolor=blue,filecolor=dullmagenta,urlcolor=darkblue} 

\def\Xint#1{\mathchoice
   {\XXint\displaystyle\textstyle{#1}}%
   {\XXint\textstyle\scriptstyle{#1}}%
   {\XXint\scriptstyle\scriptscriptstyle{#1}}%
   {\XXint\scriptscriptstyle\scriptscriptstyle{#1}}%
   \!\int}
\def\XXint#1#2#3{{\setbox0=\hbox{$#1{#2#3}{\int}$}
     \vcenter{\hbox{$#2#3$}}\kern-.5\wd0}}

\def\dashint{\Xint-}

\newtheorem{theorem}{Theorem}[section]
\newtheorem*{theorem*}{Theorem}
\newtheorem{lemma}[theorem]{Lemma}
\newtheorem*{lemma*}{Lemma}
\newtheorem{proposition}[theorem]{Proposition}
\newtheorem{corollary}[theorem]{Corollary}

\theoremstyle{definition}

\theoremstyle{remark}
\newtheorem{remark}[theorem]{Remark}
\newtheorem{question*}[theorem]{Question}

\numberwithin{equation}{section}

\theoremstyle{theorem}
\newtheorem{ltheorem}{Theorem}

\newcommand{\esssup}{\operatorname{ess \, sup}}

\begin{document}


\null

\vskip-40pt

\null

\title[Nondoubling Calder\'on-Zygmund theory | A dyadic approach]{Nondoubling Calder\'on-Zygmund theory \\ |a dyadic approach|}

\author{Jos\'e M. Conde-Alonso and Javier Parcet}


\addtolength{\parskip}{+1ex}

\maketitle

\null

\vskip-55pt

\null

\begin{abstract}
Given a measure $\mu$ of polynomial growth, we refine a deep result by David and Mattila to construct an atomic martingale filtration of $\mathrm{supp}(\mu)$ which provides the right framework for a dyadic form of nondoubling harmonic analysis. Despite this filtration being highly irregular, its atoms are comparable to balls in the given metric  |which in turn are all doubling| and satisfy a weaker but crucial form of regularity. Our dyadic formulation is effective to address three basic questions:
\begin{itemize}
\item[i)] A dyadic form of Tolsa's RBMO space which contains it.

\item[ii)] Lerner's domination and $A_2$-type bounds for nondoubling measures.

\item[iii)] A noncommutative form of nonhomogeneous Calder\'on-Zygmund theory.
\end{itemize}
Our martingale RBMO space preserves the crucial properties of Tolsa's original definition and reveals its interpolation behavior with the $L_p$ scale in the category of Banach spaces, unknown so far. On the other hand, due to some known obstructions for Haar shifts and related concepts over nondoubling measures, our pointwise domination theorem via sparsity naturally deviates from its doubling analogue. In a different direction, matrix-valued harmonic analysis over noncommutative $L_p$ spaces has recently produced profound applications. Our analogue for nondoubling measures was expected for quite some time. Finally, we also find a dyadic form of the Calder\'on-Zygmund decomposition which unifies those by Tolsa and L\'opez-S\'anchez/Martell/Parcet.   
\end{abstract}

\null

\vskip-30pt

\null

\section*{\bf Introduction}\label{RBMO.Intro}

Originally inspired by the analysis of the Cauchy transform over non-Lipschitz curves in the complex plane, nonhomogeneous harmonic analysis has received much attention in the last twenty years. The theory has evolved in many different directions and nowadays a large portion of classical Calder\'on-Zygmund theory has been successfully transferred. Cotlar's inequality, $L_p$ boundedness of singular integrals, the $\mathrm{BMO}$ class of John and Nirenberg, or $Tb$ theorems have nonhomogeneous counterparts when the underlying metric measure space $(\Omega,d,\mu)$ is only assumed to have polynomial growth. This means that balls are only assumed to satisfy the inequality $\mu(B(x,r)) \lesssim \hskip1pt r^n$ for all $x \in \Omega$, every $r>0$ and some positive number $n$. A very nice comprehensive survey on the subject can be consulted in \cite{eiderman-volbergSurvey}. In a different direction, dyadic techniques and more general probabilistic methods are among the most effective tools in the homogeneous setting, where the doubling condition $\mu(B(x,2r)) \lesssim \mu(B(x,r))$ holds for all balls in the space. Dyadic maximal and dyadic square functions usually present a better behavior than their centered analogues due to their martingale nature, which has been largely exploited for doubling measures. Moreover, in the last decade or so we also find in the literature new and simple dyadic operators whose properties have deep consequences in Calder\'on-Zygmund theory. Indeed, Petermichl discovered in \cite{petermichlPhD} that the Hilbert transform can be represented as an average of simpler dyadic operators called Haar shifts. Its implication in the so-called $A_2$ conjecture led Hyt\"onen to generalize this representation theorem to other Calder\'on-Zygmund operators \cite{hytonenA2}. Slightly earlier, Lerner also introduced in \cite{lerner2010} his profound median oscillation formula from which he later discovered a surprising dyadic domination principle for Calder\'on-Zygmund operators, to which we shall go back later. Lerner's results go far beyond the fact that they yield a very simple approach to the $A_2$ conjecture. Unfortunately, despite some isolated exceptions like the remarkable use of random dyadic lattices by Nazarov, Treil and Volberg in \cite{ntvTb} or the recent analysis of Haar shifts in \cite{lopezsanchez-martell-parcet2014}, dyadic and more elaborated probabilistic techniques have not been explored systematically in the context of nondoubling measures. 

The purpose of this paper is to provide the tools for a dyadic form of nondoubling harmonic analysis and to answer some basic questions in the theory as applications of this approach. More precisely, our main motivations and results are the following:

\noindent \textbf{A. Dyadic filtrations.} Let $(\Omega,\mu)$ be a $\sigma$-finite measure space equipped with an increasing filtration of $\sigma$-subalgebras $\Sigma_1, \Sigma_2, \Sigma_3, \ldots$ whose union is dense in $(\Omega,\mu)$. Let us write $\mathsf{E}_k$ for the corresponding set of conditional expectations. The concept of regularity for martingale filtrations imposes that $\mathsf{E}_kf \lesssim \mathsf{E}_{k-1}f$ for all measurable $f: \Omega \to \mathbb{C}$. The similarity with the doubling condition is clear and in fact regularity reduces to dyadic doublingness for dyadic lattices. An elementary observation prior to developing dyadic harmonic analysis over nondoubling measures is that doublingness and dyadic doublingness are certainly very different concepts. Indeed, despite the existence of many doubling balls in measure spaces with polynomial growth, dyadically doubling lattices are not to be expected in general. A more reasonable aim is to construct an atomic martingale filtration whose atoms are comparable to balls, which in turn are all doubling. Any such filtration will generally be highly irregular, but fortunately this is not a serious obstruction for most of the key $L_p$ martingale inequalities. In addition, we could also hope for an alternative form of regularity more adapted to the necessities in harmonic analysis. This is made precise in theorem A below, which confirms and makes rigorous the belief |somehow reflected in our recent work \cite{condealonso-mei-parcet,condealonso-parcet}| on a strong relationship between nondoubling measures and irregular filtrations satisfying the conditions above. Recall that a ball $B(x,r)$ is called $(\alpha,\beta)$-doubling when $\mu(B(x,\alpha r)) \le \beta \mu(B(x,r))$.  

\begin{ltheorem} \label{RBMO.theoremA}
Let $\mu$ be a measure of $n$-polynomial growth on $\mathbb{R}^d$. Then there exist positive constants $\alpha,\beta > 100$ and a two-sided filtration 
$\{\Sigma_k: k \in \mathbb{Z}\}$ of atomic $\sigma$-algebras of $\mathrm{supp}(\mu)$ that satisfy the following properties, where $\Pi(\Sigma)$ denotes the set of atoms in the filtration$\hskip1pt:$
\begin{itemize}
\item[\emph{i)}] The $\sigma$-algebras $\Sigma_k$ are increasingly nested.

\item[\emph{ii)}] The union of $L_\infty(\mathbb{R}^d, \Sigma_k, \mu)$ is weak-$*$ dense in $L_\infty(\mu)$.

\item[\emph{iii)}] If $Q \in \Pi(\Sigma)$, there exists an $(\alpha,\beta)$-doubling ball $B_Q$ with $B_Q \subset Q \subset 28B_Q$.

\item[\emph{iv)}] If $x \in Q \in \Pi(\Sigma)$, then $$R \ = \bigcap_{\begin{subarray}{c} S \in \Pi(\Sigma) \\ S \supsetneq Q \end{subarray}} S \qquad \Rightarrow \qquad \int_{\alpha B_R\setminus 56B_Q} \frac{d\mu(y)}{|x-y|^n}  \lesssim_{n,d,\alpha,\beta} 1.$$
\end{itemize}
\end{ltheorem}

\vskip-8pt 

Our proof is based on a refinement of a remarkable construction by David and Mattila \cite{david-mattila} that we shall modify to fit our needs. As pointed above, the fact that all atoms in $\Pi(\Sigma)$ are doubling comes at the price of a highly irregular filtration $\Sigma$ which is far from being dyadically doubling. Remarkably, property iv) yields a weaker form of regularity |key for our applications| considered for the first time by Tolsa \cite{tolsaRBMO}. One could say that this construction plays the role in the nonhomogeneous setting of the dyadic lattice for the Lebesgue measure.

\vskip3pt

\noindent \textbf{B. Dyadic BMO spaces.} A BMO space is a set of functions which enjoy bounded mean oscillation in a certain sense. Both mean and oscillation can be measured in many different ways. Most frequently we find BMO spaces referred to averages over balls in a metric measure space, but we may replace these averages by conditional expectations with respect to a martingale filtration or even by the action of a nicely behaved semigroup of operators. The relationship between metric and martingale BMO spaces is well-understood for doubling measures. Indeed, expanding ideas from Garnett and Jones \cite{garnett-jones}, the metric BMO space is equivalent to a finite intersection of martingale BMO spaces over dyadic two-sided filtrations whose atoms look like balls. John-Nirenberg inequalities and the Fefferman-Stein duality theorem still hold for these larger BMO spaces of martingales, which also serve as an interpolation endpoint for the $L_p$ scale. In the nonhomogeneous setting and for certain nondoubling measures satisfying some weak concentration at the boundary, martingale filtrations have provided a framework to define satisfactory BMO spaces \cite{condealonso-mei-parcet}. Our aim here is to construct a dyadic form of Tolsa's RBMO space \cite{tolsaRBMO} |a nonhomogeneous analogue of BMO for measures of polynomial growth| which contains it and enjoys the fundamental properties above. 

Tolsa's RBMO space is very satisfactory since it yields endpoint estimates for singular integrals and also $L_p$ estimates via interpolation of operators. Surprisingly, it seems to be still open the problem of whether $\mathrm{RBMO}(\mu)$ gives the desired interpolation results with the $L_p$ scale in the category of Banach spaces. Our dyadic $\mathrm{RBMO}$ space  |$\mathrm{RBMO}_\Sigma(\mu)$ in what follows| is nothing but the martingale BMO space defined over the filtration constructed in theorem A. It models $\mathrm{RBMO}(\mu)$ in a way similar to how dyadic $\mathrm{BMO}$ models the classical $\mathrm{BMO}$ space. 

\begin{ltheorem} \label{RBMO.theoremB}
The space $\mathrm{RBMO}_{\Sigma}(\mu)$ satisfies$\hskip1pt :$
\begin{itemize}
\item[\emph{i)}] $\mathrm{RBMO}(\mu) \subset \mathrm{RBMO}_{\Sigma}(\mu)$.
\item[\emph{ii)}] John-Nirenberg inequality and Fefferman-Stein duality theory.
\item[\emph{iii)}] Interpolation in the category of Banach spaces $$[\mathrm{RBMO}_\Sigma(\mu), L_1(\mu)]_{\frac{1}{p}} = L_p(\mu) \quad \mbox{for} \quad 1<p<\infty.$$ We also have $[L_\infty(\mu),\mathrm{H}_\Sigma^1(\mu)]_{\frac{1}{p}} = L_p(\mu)$ for the predual $\mathrm{H}_\Sigma^1(\mu)$ of $\mathrm{RBMO}_\Sigma(\mu)$. 
\end{itemize}
\end{ltheorem}

\vskip-5pt

An immediate consequence of properties i) and iii) above is that Tolsa's RBMO enjoys the same interpolation properties, which solves the problem recalled above. A curiosity |which we shall justify in the more general context of theorem D below| is that Calder\'on-Zygmund extrapolation holds for kernels satisfying the H\"ormander condition, instead of the stronger and commonly used Lipschitz smoothness. In addition, theorems A and B are not limited to the Euclidean setting $(\mathbb{R}^d,\mu)$ but also hold for any upper doubling metric measure space $(\Omega,\mu)$, see remark \ref{RBMO.metricspaces}. Properties ii) and iii) follow from the martingale nature of $\mathrm{RBMO}_\Sigma(\mu)$, while property i) exploits the structure of $\Sigma$. It is precisely property i) what makes $\mathrm{RBMO}_\Sigma(\mu)$ special among martingale $\mathrm{BMO}$ spaces because Calder\'on-Zygmund operators map $L_\infty(\mu)$ into $\mathrm{RBMO}_\Sigma(\mu)$. 

\vskip5pt

\noindent \textbf{C. Oscillation and dyadic domination.} Consider an atomic filtration $\{\Sigma_k : k \in \mathbb{Z}\}$ in $(\Omega,\mu)$ whose atoms |measurable sets that are not decomposable into smaller measurable sets| are denoted by $\Pi(\Sigma)$. Given $Q \in \Pi(\Sigma)$, let us write $\langle f \rangle_Q = \dashint_Q f d\mu = \mu(Q)^{-1} \int_Q f d\mu$ for the $\mu$-average of $f$ over $Q$. Working with Euclidean spaces, the  Lebesgue measure and the standard dyadic filtration |so that $\Pi(\Sigma)$ becomes the set of all dyadic cubes| Haar shift operators take the form $$f \mapsto \sum_{Q,R \in \Pi(\Sigma)} \alpha_{QR} \hskip1pt \langle f \rangle_Q \hskip1pt \chi_R$$ with $\alpha_{QR} = 0$ when the generations of $Q$ and $R$ are very far apart. This difference between generations is called the complexity of the operator. Lerner's median oscillation formula \cite{lerner2010} relates the value of a general function $f$ at a point $x$ to its oscillations on a particularly nice family of dyadic cubes containing $x$. This led to a highly unexpected upper bound for Calder\'on-Zygmund operators |first in norm \cite{lernerPointwise} and pointwise afterwards \cite{condealonso-rey2014,lerner-nazarov2014}, see also the recent papers \cite{lacey2015,lerner2016} for more general operators| in terms of very simple dyadic operators $$|Tf(x)| \, \lesssim \, \sum_{j \in \mathcal{J}} \mathcal{A}_j |f|(x) \qquad \mbox{with} \qquad \mathcal{A}_jf = \sum_{Q \in \mathcal{S}_j} \langle f \rangle_Q \chi_Q,$$ where $\mathcal{J}$ in the above formula is finite and the sets of cubes $\mathcal{S}_j$ are sparse families |whose definition is given in section \ref{RBMO.section3}|  which depend on $T$ and $f$. 
This dyadic domination principle is so accurate that it encodes all the smoothness subtleties of Calder\'on-Zygmund operators by the sparseness of certain positive dyadic operators. Beyond its intrinsic depth, it also yields what can be arguably considered the simplest approach to the $A_2$ theorem. The first result in this direction for nondoubling measures is due to H\"anninnen \cite{hanninen2015}, who has recently extended Lerner's oscillation formula, but his approach was too limited to provide pointwise domination of nondoubling Calder\'on-Zygmund operators by sparse shifts or to explore related weighted inequalities. 

Our dyadic approach here is simple enough to deduce the pointwise dyadic domination and related $A_2$ bounds from H\"anninen's oscillation formula. Both results hold for standard nondoubling Calder\'on-Zygmund operators |whose definition is postponed to the body of the paper| in terms of sparse operators adapted to our filtration $\Sigma$ from theorem A. However, as we just indicated above, dyadic covering lemmas are not available in our setting. Also, the negative results in \cite{lopezsanchez-martell-parcet2014} suggest that even if they were available, boundedness of high complexity dyadic shifts by sparse operators is not to be expected. Therefore, we are forced to introduce a modification in our formula which includes a nicely behaved maximal operator $$\mathcal{M}^cf(x) = \sup_{r>0} \frac{1}{\mu(B(x,5r))} \int_{B(x,r)} |f| \;d\mu.$$

\begin{ltheorem} \label{RBMO.theoremC}
Let $\mu$ be a measure of $n$-polynomial growth on $\mathbb{R}^d$ and let $\Sigma = \{\Sigma_k : k \in \mathbb{Z} \}$ the atomic filtration from theorem \emph{A}. Given a Calder\'on-Zygmund operator $T$, the following pointwise estimate holds for certain $\mu$-sparse family $\mathcal{S} \subset \Pi(\Sigma)$ $$Tf(x) \lesssim_{d,n,\mu} \sum_{Q \in \mathcal{S}} \left[ \inf_{y \in Q} \mathcal{M}^c f(y) \right] \chi_Q (x).$$ In particular, if $w$ is an $A_2$ weight with respect to $\mu$, we have $$\big\| T: L_2(wd\mu) \to L_2(wd\mu) \big\| \lesssim_{d,n,\mu} [w]_{A_2(\mu)}^2.$$
\end{ltheorem}

At this point, it is important to note that a stronger $A_2$ bound |linear in the $A_2$-characteristic in line with the doubling $A_2$ theorem| was announced in \cite[theorem 5.1]{lacey2015} as the result of a personal communication between Treil and Volberg with Lacey. At the time of this writing and after several contacts with Lacey and Volberg, it seems that such assertion was based on a pointwise domination theorem which was not properly justified, and therefore the announced $A_2$-bound cannot be considered to be proven so far. Besides the obstruction to use high complexity Haar shifts, Lacey's approach \cite{lacey2015} and subsequent work \cite{lerner2016} do not immediately generalize to the nondoubling setting. In particular, it is not absolutely clear that Theorem C above is suboptimal.  

\noindent \textbf{D. Matrix-valued harmonic analysis.} Calder\'on-Zygmund operators which act on matrix-valued functions were investigated for the first time in \cite{parcet2009}. In fact, the aim for estimates independent of the matrix size allows to replace matrices by arbitrary von Neumann algebras at almost no cost. This particular instance of noncommutative Calder\'on-Zygmund theory |semicommutative algebras over Euclidean spaces equipped with the Lebesgue measure| has been the main tool in the recent solution of the Nazarov-Peller conjecture \cite{Sukochev} and it has also served to obtain pioneering results on smooth Fourier multipliers with nonabelian frequencies \cite{junge-mei-parcetSmooth}. In addition, it provides the optimal behavior of $L_p$ constants which improves Bourgain's UMD approach. These are strong applications which naturally motivate the study of nondoubling Calder\'on-Zygmund operators acting on matrix-valued functions. In fact, both endpoint estimates for $p=1,\infty$ have been open for quite some time. Potential applications of it point towards the convergence of Fourier series which are frequency supported by free groups. The noncommutative setting goes beyond vector-valued harmonic analysis and introduces a number of genuine new difficulties from the lack of points after quantization, details can be found in \cite{junge-mei-parcetSmooth,parcet2009}. At this point, the theory of noncommutative martingale $L_p$ inequalities |nowadays well-stablished| becomes crucial in conjunction with our martingale approach towards nondoubling harmonic analysis. This combination yields an operator valued form of $\mathrm{RBMO}_\Sigma(\mu)$ |denote by $\mathrm{RBMO}_\Sigma(\mathcal{A})$ in what follows| which admits the expected endpoint estimate for $p=\infty$, so that Calder\'on-Zygmund extrapolation holds by interpolation and duality. The definition of Calder\'on-Zygmund operator in this context |for which we impose noncommutative forms of the size and H\"ormander kernel conditions| and other notions from operator algebra will be given in the body of the paper.

\begin{ltheorem} \label{RBMO.theoremD}
Let $\mu$ be a measure of $n$-polynomial growth on $\mathbb{R}^d$ and let $\Sigma = \{\Sigma_k : k \in \mathbb{Z} \}$ be the atomic filtration from theorem \emph{A}. Let $\mathcal{A} = L_\infty(\mu) \bar\otimes \mathcal{M}$ for some noncommutative measure space $(\mathcal{M},\tau)$ and equip $\mathcal{A}$ with its natural trace $\varphi = \mu \otimes \tau$. Then, every $L_2$-bounded Calder\'on-Zygmund operator satisfies $T: \mathcal{A} \to \mathrm{RBMO}_{\Sigma}(\mathcal{A})$. Moreover, $L_p$-boundedness follows for $1 < p < \infty$ by interpolation and duality, which also hold in this case.
\end{ltheorem}

Finally, the properties of the martingale filtration $\Sigma$ yield a dyadic form of Calder\'on-Zygmund decomposition for functions in $L_1(\mu)$.  It resembles the classical one although the irregularity of $\Sigma$ forces to subtract the average over the ancestors of the maximal cubes instead of the actual maximal cubes. This is very similar to the Calder\'on-Zygmund-Gundy decomposition in \cite{lopezsanchez-martell-parcet2014} but goes beyond it, since we can deduce the weak type $(1,1)$ bound for Calder\'on-Zygmund operators from it. In summary, this new Calder\'on-Zygmund type decomposition |theorem \ref{RBMO.CZ} in the text| unifies the decompositions in \cite{lopezsanchez-martell-parcet2014,tolsaWeakType} and admits potential applications in the noncommutative setting.

\vskip5pt

\noindent \textbf{Acknowledgement.} The authors thank Xavier Tolsa for bringing to our attention David and Mattila's construction and for several fruitful discussions on the content of this paper. Both authors were partially supported by CSIC Project PIE 201650E030 and also by ICMAT Severo Ochoa Grant SEV-2015-0554, and the first named author was supported in part by ERC Grant 32501. 


\section{\bf A lattice of cubes and a doubling nonregular filtration}
\label{RBMO.sectiondavidmattila}

For the rest of the paper, fix a Radon measure $\mu$ on $\mathbb{R}^d$ of normalized  $n$-polynomial growth, that is, one for which $\mu(B(x,r)) \leq r^n$ for all $x \in \mathbb{R}^d$ and $r>0$. Even though we work on $\mathbb{R}^d$, we will only be concerned with what happens on the support of the measure $\mu$. Thus, we will slightly abuse notation denoting by $B$ the intersection of the set $B$ with $\mathrm{supp}(\mu)$. Another way of thinking about this is to assume that whenever we write a relation between sets
, we mean that it holds up to a set of $\mu$-measure $0$. To build the doubling filtration $\Sigma$ as stated in the introduction, we will modify a construction that is due to  David and Mattila \cite[theorem 3.2]{david-mattila}. We give the precise statement that we need below, the main novelty being the presence of arbitrarily large cubes and the ``lack of quadrants", which means that whenever we take the union of all the ancestors of a given cube in the lattice we get the full support of the measure $\mu$. This particular feature will be necessary to ensure that the space of $\mathrm{BMO}$ type constructed in section \ref{RBMO.sectionRBMO} interpolates with $L_p(\mu)$.

\begin{proposition}\label{RBMO.davidmattila}
Let $\mu$ be of $n$-polynomial growth on $\mathbb{R}^d$ and fix $\alpha \geq 100$ and $\beta=\alpha^{\ell}$ for some $\ell \geq d+1$. Then, there exist $A \gg \beta$ of the form $A=\alpha^{\ell m}$ and a sequence of partitions $\mathscr{D}=\{\mathscr{D}_k\}_{k\in \mathbb{Z}}$ of $\mathrm{supp}(\mu)$, whose elements are called cubes, with the following properties:
\begin{itemize}
\item[\emph{i)}]The partitions are nested, so that $Q \cap R \in \{\emptyset,Q,R\}$ for any pair of cubes $Q,R$.
\item[\emph{ii)}] For each $k \in \mathbb{Z}$ and $Q \in \mathscr{D}_k$, there exists $x_Q \in \mathrm{supp}(\mu)$ and a radius $r(Q) \sim_{\beta} A^{-k}$ such that the balls $B_Q=B(x_Q,r(Q))$ satisfy:
\begin{itemize}
\item[$\bullet$] $B_Q \cap \mathrm{supp}(\mu) \subset Q \subset 28B_Q \cap \mathrm{supp}(\mu)$.
\item[$\bullet$] The balls $5B_Q$ associated to cubes of the same generation are pairwise disjoint.
\end{itemize}
\item[\emph{iii)}] For each $Q$, either $B_Q$ is $(\alpha,\beta)$-doubling or else $\mu(\alpha B_Q) < \beta^{-i} \mu(\alpha^{i+1}B_Q)$ whenever $\alpha^{i} \leq \beta$.
\item[\emph{iv)}] For every $k\in \mathbb{Z}$ and every $Q\in \mathscr{D}_k$, 
$$
\mathrm{supp}(\mu) \subset \bigcup_{\begin{subarray}{c} R \; \mathrm{cube} \\R \supset Q \end{subarray}} R.
$$
\end{itemize}
\end{proposition}

In the proof of proposition \ref{RBMO.davidmattila} we shall use the $5R$ covering lemma allowing ourselves to choose a particular ball. The modification of the proof is trivial, but we record the statement precisely for the sake of clarity.

\begin{lemma}\label{RBMO.vitali}
Let $A \subset \mathbb{R}^d$ be a set such that each $x\in A$ has an associated ball $B_x$ with radius $r_x$. Pick some $x_0 \in A$ such that $r_{x_0}\geq \frac12 \sup_{x\in A} r_x$. Then, there exists a subcollection $\mathcal{B}=\{B_j\}_{j\in \mathbb{N}} \subset \{B_x\}_{x\in A}$ satisfying:
\begin{itemize}
\item The balls in $\mathcal{B}$ are pairwise disjoint.
\item $A \subset \cup_j 5B_j$.
\item $B_{x_0} \in \mathcal{B}$.
\end{itemize}
The family $\{B_j\}_j$ is called the $5R$-covering of $A$ associated to $x_0$. 
\end{lemma}

\begin{proof}[Proof of proposition \ref{RBMO.davidmattila}] The proof of theorem 3.2 of \cite{david-mattila} is highly technical and long. Our version uses most of the steps of its proof |although we cannot apply the theorem directly|, and therefore we will only explain the relevant changes to avoid the repetition of arguments that we do not modify. First choose the constant $A$ as the smallest constant of the form $\alpha^{m\ell}$ that is greater than the constant $A_0$ that appears in theorem 3.2 of \cite{david-mattila}. Since $\mu$ has $n$-polynomial growth, for $\mu$-almost every point $x$, the sequence $\{B(x,\frac{7}{8} \beta \alpha^j)\}_{j\in \mathbb{N}}$ contains infinitely many $(\alpha,\beta)$-doubling balls. Fix one such point $x_0$. There exists $i \in \{0,1,\ldots,m\ell-1\}$ such that the subsequence composed of balls whose radii $r=\frac{7}{8} \beta \alpha^j$ satisfy 
$$
j \; \mathrm{mod} \; (m\ell) =i
$$
contains infinitely many $(\alpha,\beta)$-doubling balls. Now, for each generation $k \in \mathbb{Z}$, we proceed as in the proof of \cite[theorem 3.2]{david-mattila}, except for the point $x_0$:
\begin{itemize}
\item For each $k \in \mathbb{Z}$ and $x \in \mathrm{supp}(\mu)$, $x\not=x_0$, we choose a radius $r_k(x)$ such that
\begin{equation*}
\alpha^i A^{-k} \leq r_k(x) \leq \beta \alpha^i A^{-k}
\end{equation*}
according to the following algorithm: either $r_k(x)$ is chosen so that $B(x,r_k(x))$ is $(\alpha,\beta)$-doubling or, if there are no $(\alpha,\beta)$-doubling balls with radii satisfying the equation above, we take $r_k(x)= \alpha^i A^{-k}$.
\item For $x=x_0$, if the ball $B(x_0,\frac{7}{8} \beta \alpha^i A^{-k})$ is $(\alpha,\beta)$-doubling, we choose $r_k(x_0) = \frac{7}{8} \beta \alpha^i A^{-k}$. Notice that, by the above reasoning, we make this choice infinitely many times. If 
$$
B(x_0,\frac{7}{8} \beta \alpha^i A^{-k})
$$ 
is not $(\alpha,\beta)$-doubling, we choose $r_k(x_0)$ as in the previous item.
\end{itemize}

We now fix $k$, and apply lemma \ref{RBMO.vitali} to $\{B(x,5r_k(x))\}_{x \in \mathrm{supp}(\mu)}$ to obtain the $5R$-covering of $\mathrm{supp}(\mu)$ associated to $x_0$, whenever it is possible. This happens infinitely many times by our choices of $r_k(x_0)$. Otherwise, we just apply lemma \ref{RBMO.vitali} without specifying the point.
In either case, the resulting family is the family of disjoint balls $5B_Q$ associated with the cubes of generation $k$. The rest of the proof of properties $(i)-(iii)$ is exactly that of theorem 3.2 in \cite{david-mattila}. Notice that the arguments there do not depend on the size of the cubes being uniformly bounded. This yields a collection $\mathscr{D}=\cup_{k \in \mathbb{Z}} \mathscr{D}_k$ with the mentioned properties. It only remains to show that our construction does not have ``quadrants". But this is easy: given any cube $Q \in \mathscr{D}$, there exists some $R\in \mathscr{D}$ whose associated ball $B_R=B_R(x_0,r(x_0))$ intersects it, because there are arbitrarily large balls of this kind centered at $x_0$. Thus, the claim trivially follows. 
\end{proof}

\begin{remark}
By the proof of theorem 3.2 of \cite{david-mattila}, the cubes in $\mathscr{D}$ have small boundaries: for each $Q\in \mathscr{D}_k$ and $i \in \mathbb{Z}_+$, define
$$
N_{i}^{\mathrm{ext}}(Q)= \{x \in \mathrm{supp}(\mu)\setminus Q: \mathrm{dist}(x,Q) < A^{-k-i}\},
$$
$$
N_{i}^{\mathrm{int}}(Q)= \{x \in Q: \mathrm{dist}(x,\mathrm{supp}(\mu)\setminus Q) < A^{-k-i}\}.
$$
Then $\mu(N_{i}^{\mathrm{ext}}(Q) \cup N_{i}^{\mathrm{int}}(Q)) \leq (c_d \beta^{-3d-1}A)^{-i} \mu(90B_Q)$.
\end{remark}

\begin{remark}
By a modification of the arguments above one may force the apparition of quadrants in the construction of proposition \ref{RBMO.davidmattila}. A quadrant is a proper subset $E$ of $\mathbb{R}^d$ such that a cube $Q\in \mathscr{D}$ that is contained in $E$ has all of its ancestors lying inside $E$; or, equivalently, a proper subset of $\mathbb{R}^d$ that is a union of sets in $\mathscr{D}_k$ for all $k\in \mathbb{Z}$. This is clearly seen, for example, in the case of the real line with the Gaussian measure $\gamma$, the probability measure with density given by 
$$
d\gamma(x) = \frac{1}{\sqrt{2\pi}} e^{-|x|^2/2}dx.
$$
$\gamma$ is nondoubling, but all intervals of the form $[0,a]$ are $(2,2)$-doubling. Therefore, if in the proof we choose the sequence of centers $\{A^k\}_{k\in \mathbb{Z}_+}$ instead of the point $x_0$, and then we repeat the steps of the proof, we will get that the intervals $[0,2A^k]$, for $k\in \mathbb{Z}_+$, belong to the resulting system $\mathscr{D}$. This divides the real line in the usual two quadrants. 
\end{remark}

Using proposition \ref{RBMO.davidmattila} we may now build the filtration $\Sigma$ and prove theorem \ref{RBMO.theoremA}. We say that a cube $Q\in \mathscr{D}$ is $(\alpha,\beta)$-doubling if its associated ball $B_Q$ is $(\alpha,\beta)$-doubling.

\begin{proof}[Proof of theorem \ref{RBMO.theoremA}]
The basic idea of the construction is to exploit the abundance of doubling cubes in $\mathscr{D}$ by constructing a filtration only with them. By lemma 5.28 in \cite{david-mattila}, given any cube $R\in\mathscr{D}$, there exists a (pairwise disjoint) collection of cubes $\{Q_i\} \subset \mathscr{D}$ which are $(\alpha,\beta)$-doubling and such that
\begin{equation} \label{RBMO.existendoblantes}
\mu\left(R \setminus \bigcup_i Q_i \right) = 0.
\end{equation}
On the other hand, by the construction of proposition \ref{RBMO.davidmattila}, any cube $Q \in \mathscr{D}$ is contained in an $(\alpha,\beta)$-doubling cube $R\in \mathscr{D}$. Therefore, we may construct the filtration $\Sigma=\{\Sigma_k\}_{k \in \mathbb{Z}}$ as follows:

\begin{itemize}
\item We start with any $(\alpha,\beta)$-doubling cube $Q_0 \in \mathscr{D}$. We proclaim that $Q_0 \in \Pi(\Sigma_0)$.
\item The parent (in $\Sigma$) $\widehat{Q}_0$ of $Q_0$ is the smallest $(\alpha,\beta)$-doubling cube of $\mathscr{D}$ that contains (properly) $Q_0$. Notice that $\widehat{Q}_0$ exists by item (iv) of proposition \ref{RBMO.davidmattila}. We proclaim that $\widehat{Q}_0$ belongs to $\Pi(\Sigma_{-1})$.
\item Inductively, for $j\geq 1$ we define the $(j+1)$-th ancestor of $Q_0$ as the smallest $(\alpha,\beta)$-doubling cube of $\mathscr{D}$ that contains (properly) the $j$-th ancestor of $Q_0$. We proclaim that it belongs to $\Pi(\Sigma_{-j-1})$. Notice that the union of the $j$-th ancestors of $Q_0$, $j\geq 1$, covers the whole space $\mathrm{supp}(\mu)$.
\item Given a cube $Q\in\Pi(\Sigma)$, its children (in $\Sigma$) are the maximal $(\alpha,\beta)$-doubling cubes of $\mathscr{D}$ that are properly contained in $Q$. These always cover $Q$ by \eqref{RBMO.existendoblantes}. Notice that our definition is consistent: the $j$-th ancestor of $Q_0$ is always a child of the $(j+1)$-th ancestor of $Q_0$. In general, we say that a cube $R\in \Pi(\Sigma)$ belongs to $\Sigma_k$ if its parent $\widehat{R}$ belongs to $\Pi(\Sigma_{k-1})$. 
\end{itemize}

The previous items define the whole filtration $\Sigma=\{\Sigma_k\}_{k\in \mathbb{Z}}$. Also, notice that $\Pi(\Sigma) \subset \mathscr{D}$. The construction immediately yields the properties (i) and (iii) in the statement of theorem \ref{RBMO.theoremA}. Also, the Lebesgue differentiation theorem holds for sets in $\Sigma$, which means that for $\mu$-a.e. $x\in \mathrm{supp}(\mu)$, 
$$
f(x) = \lim_{\begin{subarray}{c} x \in Q \in \Pi(\Sigma) \\ r(B_Q) \to 0 \end{subarray}} \langle f \rangle_Q \;\; \mathrm{if} \; f\in L_1^{\mathrm{loc}}(\mu).
$$
This implies that $\cup_k L_\infty(\mathbb{R}^d,\Sigma_k,\mu)$ is weak-$*$ dense in $L_\infty(\mu)$, so we have (ii). To prove (iv), we need to introduce some notation about ancestors of general cubes $Q$, and not only for the $Q_0$ (which will no longer play an important role). These will be used throughout the rest of the paper. If $k_0<k_1$, we will denote the unique cube $R \in \mathscr{D}_{k_0}$ which contains $Q \in \mathscr{D}_{k_1}$ by $Q^{(k_1-k_0)}$. We will also denote by $\widehat{Q}$ the smallest cube in $\Sigma$ that properly contains $Q \in \Sigma$, and write $\mathrm{gen}(T) = k$ if $T\in \mathscr{D}_k$. The proof of (iv) now follows from the following observation, which constitutes lemma 5.31 in \cite{david-mattila}: if $Q \subset R$ are cubes in $\mathscr{D}$ and if all intermediate cubes $T \in \mathscr{D}$ (that is, all $T \in \mathscr{D}$ such that $Q \subsetneq T \subsetneq R$) are not $(\alpha,\beta)$-doubling, then 
\begin{equation}\label{RBMO.keyproperty}
\mu(\alpha B_T) \leq A^{-10n(\mathrm{gen}(T)-\mathrm{gen}(R)-1)}\mu(\alpha B_R).
\end{equation}
Indeed, taking into account \eqref{RBMO.keyproperty}, we may compute for $x\in Q$
\begin{eqnarray*}
 \int_{\alpha B_{\widehat{Q}} \setminus 56B_Q} \frac{1}{|x-y|^n}d\mu(y)
& \lesssim & 1+\sum_{j=1}^{\mathrm{gen}(Q)-\mathrm{gen}(\widehat{Q})}  \int_{\alpha B_{Q^{(j)}} \setminus \alpha B_{Q^{(j-1)}}} \frac{1}{|x-y|^n}d\mu(y) \\
& \lesssim & 1+\sum_{j=1}^{\mathrm{gen}(Q)-\mathrm{gen}(\widehat{Q})} \frac{\mu(\alpha B_{Q^{(j)}}) }{r(B_{Q^{(j-1)}})^n} \\
& \leq & 1+\sum_{j=1}^{\mathrm{gen}(Q)-\mathrm{gen}(\widehat{Q})} A^{n\;\mathrm{gen}(Q^{(j-1)})} A^{-10n(\mathrm{gen}(Q^{(j)})-\mathrm{gen}(\widehat{Q}) -1)} \mu(\alpha B_{\widehat{Q}}) \\
& \lesssim & 1+\mu(\alpha B_{\widehat{Q}}) A^{n\;\mathrm{gen}(\widehat{Q})} \sum_{j=1}^{\mathrm{gen}(Q)-\mathrm{gen}(\widehat{Q})} A^{-9nj} \lesssim 1. \\
\end{eqnarray*}

This ends the proof of theorem \ref{RBMO.theoremA}. 
\end{proof}

\begin{remark}
Notice that (iv) in theorem \ref{RBMO.theoremA} is similar to a property that holds for sequences of concentric nondoubling cubes in $\mathbb{R}^d$ and was key for the construction in \cite{tolsaRBMO} (and in the weak $(1,1)$ inequality in \cite{tolsaWeakType}). Here we reinterpret it as a (weak) regularity of the filtration $\Sigma$.
\end{remark}


\section{\bf A martingale BMO for nondoubling measures}
\label{RBMO.sectionRBMO}

As we explained in the introduction, we define our dyadic RBMO space, $\mathrm{RBMO}_\Sigma(\mu)$, as the martingale BMO space associated with the filtration $\Sigma$ constructed in section \ref{RBMO.sectiondavidmattila}. We choose the values $\alpha = 2\cdot 28^2$ and $\beta=\alpha^{d+1}$. Given a probability (or $\sigma$-finite) measure space $(\Omega,\nu)$ and a bilateral filtration $\mathcal{F}=\{\mathcal{F}_k\}_{k\in\mathbb{Z}}$, its associated martingale BMO space is the space of $\nu$-measurable functions $f$ with norm
$$
\|f\|_{\mathrm{BMO}} := \sup_{k \in \mathbb{Z}} \left\|\mathsf{E}_{\mathcal{F}_k} \left|f-\mathsf{E}_{\mathcal{F}_{k-1}}f \right|^2 \right\|_{L_\infty(\nu)}^{\frac12}.
$$ 
The definition of the martingale BMO norm and the classical one are somewhat different since in the martingale case we subtract $\mathsf{E}_{\mathcal{F}_{k-1}} f$ instead of $\mathsf{E}_{\mathcal{F}_k} f$. This turns out to be very relevant because it is what ensures that the martingale BMO spaces interpolate regardless of the regularity of the filtration $\mathcal{F}$. In case one subtracted $\mathsf{E}_{\mathcal{F}_k} f$ in the definition, the resulting space, denoted as $\mathrm{bmo}$, has different interpolation properties in the general case, see \cite{condealonso-mei-parcet,garsia1973} for more details and background on martingale BMO spaces. In our case $\mathcal{F}=\Sigma$ is an atomic filtration, and therefore we have the expression
\begin{eqnarray*}
\|f\|_{\mathrm{RBMO}_\Sigma} & = & \sup_{k \in \mathbb{Z}} \left\| \mathsf{E}_{\Sigma_k} \left|f-\mathsf{E}_{\Sigma_{k-1}}f \right|^2 \right\|_{L_\infty(\mu)}^{\frac12} \\
& \sim & \sup_{k \in \mathbb{Z}} \left\| \mathsf{E}_{\Sigma_k} \left|f-\mathsf{E}_{\Sigma_{k-1}}f \right| \right\|_{L_\infty(\mu)} \\
& = & \sup_{Q \in \Pi(\Sigma)} \frac{1}{\mu(Q)} \int_Q \left|f - \langle f \rangle_{\widehat{Q}} \right| d\mu. \\
\end{eqnarray*}
The equivalence of norms in the second step above follows from the John-Nirenberg inequality, which holds for all martingale BMO spaces (see below), and hence for $\mathrm{RBMO}_{\Sigma}(\mu)$. We have arrived to an expression which we can compare to another one for the $\mathrm{RBMO}$ norm of Tolsa, as we shall see now. Given two balls $B_1,B_2$ in $\mathbb{R}^d$, we define 
$$
K_{B_1,B_2}:= 1+ \sum_{j=0}^{N_{B_1,B_2}} \frac{\mu(2^{j}B_1)}{r(2^jB_1)^n},
$$
where $N_{Q,R}$ is the smallest positive integer $\ell$ such that $B_2\subset 2^{\ell}B_1$. According to lemma 2.10 (c) of \cite{tolsaRBMO} and the fact that the definition of the space does not change if we use balls instead of cubes, a function $f$ belongs to $\mathrm{RBMO}(\mu)$ if and only if the following quantity is finite:
$$
\|f\|_{\mathrm{RBMO}(\mu)} = \max\left\{\|f\|_{*}, \|f\|_{\mathrm{d}} \right\} <\infty,
$$ 
where
\begin{eqnarray*}
\|f\|_{*} & = & \sup_{\begin{subarray}{c}  B \; (2,\beta)-\mathrm{doubling} \\ B \; \mathrm{ball} \end{subarray}} \frac{1}{\mu(B)} \int_B \left|f- \langle f \rangle_B \right| \; d\mu, \\
\|f\|_{\mathrm{d}} & = & \sup_{\begin{subarray}{c} B_1,B_2 \; (2,\beta)-\mathrm{doubling} \\ B_1 \subset B_2\end{subarray}} \frac{\left|\langle f \rangle_{B_1} -\langle f\rangle_{B_2} \right|}{K_{B_1,B_2}}. \\
\end{eqnarray*}

Let us now justify the first item in theorem \ref{RBMO.theoremB}. We want to show that
$$
\mathrm{RBMO}(\mu) \subset \mathrm{RBMO}_{\Sigma}(\mu),
$$
that is, 
$$
\|f\|_{\mathrm{RBMO}_{\Sigma}(\mu)} \lesssim \|f\|_{\mathrm{RBMO}(\mu)}
$$
for all functions $f \in \mathrm{RBMO}(\mu)$. The proof is just a computation for the norm in $\mathrm{RBMO}_{\Sigma}(\mu)$. Fix some $k\in \mathbb{Z}$ and $Q \in \Pi(\Sigma_k)$ and split into pieces as follows:
\begin{eqnarray*}
\dashint_Q \left|f -\langle f \rangle_{\widehat{Q}} \right| d\mu & \leq & \dashint_Q \left|f -\langle f \rangle_{Q} \right| d\mu + \left|\langle f \rangle_Q - \langle f \rangle_{\widehat{Q}} \right|\\
 & \leq & 3\dashint_Q \left|f -\langle f \rangle_{28B_Q} \right| d\mu + \dashint_{\widehat{Q}} \left|f -\langle f \rangle_{28B_{\widehat{Q}}} \right| d\mu + \left|\langle f \rangle_{28B_Q} - \langle f \rangle_{28B_{\widehat{Q}}} \right| \\
 & = & \mathrm{I} + \mathrm{II} + \mathrm{III}.\\
 \end{eqnarray*}
Now, since $B_Q \subset Q \subset 28B_Q$ and we have chosen $\alpha$ so that $B_Q$ is $(\alpha,\beta)$-doubling (which in particular implies that $28B_Q$ is $(2,\beta)$-doubling), we can estimate
$$
\mathrm{I} \leq 3 \frac{1}{\mu(Q)} \int_{28B_Q} \left| f - \langle f \rangle_{28B_Q} \right| d\mu \lesssim \dashint_{28B_Q} \left| f - \langle f \rangle_{28B_Q} \right| d\mu \leq \|f\|_{\mathrm{RBMO}(\mu)},
$$
and similarly for $\mathrm{II}$. We are left with $\mathrm{III}$. On the one hand, we have
$$
B_Q \subset Q \subset \widehat{Q} \subset 28B_{\widehat{Q}},
$$ 
and therefore $28B_Q \subset 28^2B_{\widehat{Q}}$. So, we can write
$$
\left|\langle f \rangle_{28B_Q} - \langle f \rangle_{28B_{\widehat{Q}}} \right| \leq \left|\langle f \rangle_{28B_Q} - \langle f \rangle_{28^2B_{\widehat{Q}}} \right| + \left|\langle f \rangle_{28^2B_{\widehat{Q}}} - \langle f \rangle_{28B_{\widehat{Q}}} \right|.
$$
We have that the three balls $28B_Q$, $28B_{\widehat{Q}}$ and $28^2B_{\widehat{Q}}$ are $(2,\beta)$-doubling. By definition, $K_{28B_{\widehat{Q}},28^2B_{\widehat{Q}}} \lesssim 1$, which implies
$$
\left|\langle f \rangle_{28^2B_{\widehat{Q}} }- \langle f \rangle_{28B_{\widehat{Q}}} \right| \sim \frac{\left|\langle f \rangle_{28^2B_{\widehat{Q}}} - \langle f \rangle_{28B_{\widehat{Q}}} \right|}{K_{28B_{\widehat{Q}},28^2B_{\widehat{Q}}} } \leq \|f\|_{\mathrm{RBMO}(\mu)}.
$$
Finally, one can check (see again \cite{tolsaRBMO}) that
$$
K_{B_1,B_2} \sim 1 + \int_{r(B_1) \leq |y-x_{B_1}| \leq r(B_2)} \frac{1}{|y-x_{B_1}|^n} d\mu(y).
$$
Therefore, we can apply item (iv) of theorem \ref{RBMO.theoremA} to the cubes $Q$ and $\widehat{Q}$ with $x=x_{Q}$ to obtain
$$
K_{28B_{Q},28^2B_{\widehat{Q}}} \lesssim 1 + \int_{\alpha B_{\widehat{Q}} \setminus 56B_Q} \frac{1}{|x_{Q}-y|^n}d\mu(y) \lesssim 1,
$$
which also implies that 
$$
\left|\langle f \rangle_{28B_Q} - \langle f \rangle_{28^2B_{\widehat{Q}}} \right| \lesssim \|f\|_{\mathrm{RBMO}(\mu)}.
$$
The above argument provides an interpretation of the geometric coefficients $K_{B_1,B_2}$ in terms of the martingale structure that we have built. In particular, one can see that if $Q \subset R$ are such that $Q \in \Pi(\Sigma_{k_1})$ and $R \in \Pi(\Sigma_{k_2})$ then $K_{B_Q,B_R} \sim k_1 - k_2$. This justifies why we view item (iv) of theorem \ref{RBMO.theoremA} as some sort of regularity of $\Sigma$.

\begin{remark}
The inclusion $\mathrm{RBMO}(\mu) \subset \mathrm{RBMO}_\Sigma(\mu)$ can be |and generally will be| strict. An instance of this can be seen by taking the example of the Gaussian measure in $\mathbb{R}$ at the end of the previous section and considering the function $f(x) = e^{|x|^2} \log(|x|)$, see also \cite{condealonso-mei-parcet} for a similar example.
\end{remark}

Notice that our definition of $\mathrm{RBMO}_\Sigma(\mu)$ is somewhat simpler than that of $\mathrm{RBMO}(\mu)$. However, as we have just seen it is at least as powerful. However, the use of a martingale BMO norm |in which we subtract $\mathsf{E}_{\Sigma_{k-1}}$ when we average over atoms of $\Sigma_k$| yields the rest of the good properties claimed in the statement of theorem \ref{RBMO.theoremB} essentially for free. Indeed, the John-Nirenberg property for martingale spaces (see \cite{garsia1973} for the proof), that we already used, implies that
$$
\|f\|_{\mathrm{RBMO}_{\Sigma}(\mu)} \sim_p \sup_{k \in \mathbb{Z}} \left\| \mathsf{E}_{\Sigma_k} \left|f-\mathsf{E}_{\Sigma_{k-1}}f \right|^p \right\|_{\infty}^{\frac1p}, \; 1\leq p<\infty.
$$
We now focus our attention on the predual of $\mathrm{RBMO}_{\Sigma}(\mu)$, which will be denoted by $\mathrm{H}_{\Sigma}^1(\mu)$. Since $\mathrm{RBMO}_{\Sigma}(\mu)$ is a martingale $\mathrm{BMO}$ space, its predual is the subspace of $L_1(\mu)$ functions with norm given by
$$
\|f\|_{\mathrm{H}_{\Sigma}^{1}(\mu)} = \left\|\left(\sum_{k\in \mathbb{Z}} |\mathsf{D}_{k}f|^2 \right)^{\frac12} \right\|_{L_1(\mu)}.
$$
Here $\mathsf{D}_k = \mathsf{E}_{\Sigma_k} - \mathsf{E}_{\Sigma_{k-1}}$ is the $k$-th martingale difference operator. This is the standard expression of the norm in martingale $\mathrm{H}_1$ spaces. On the other hand, in \cite{tolsaRBMO} the predual of $\mathrm{RBMO}(\mu)$ was described as a space of functions decomposable into atomic blocks, a generalization of the classical atoms that span the usual Euclidean $\mathrm{H}_1$ space. A connection between the two worlds is again given by martingale theory: as was shown in \cite{condealonso-parcet}, one can find an atomic block decomposition for functions in any martingale $\mathrm{H}_1$ space. Let us briefly describe it in the setting of $\mathrm{H}_{\Sigma}^1(\mu)$: a function $b$ is said to be a $p$-atomic block, $1< p \leq \infty$, if the following conditions are satisfied:
\begin{itemize}
\item There exists some $k\in \mathbb{Z}$ such that $\mathsf{E}_{\Sigma_k}b = 0$.
\item $b= \sum_j \lambda_j a_j$, where $\lambda_j$ are scalars and $a_j$ are $L_p$ functions such that
\begin{enumerate}
\item $\mathrm{supp}(a_j) \subset A_j \in \Sigma_{k_j}, \; k_j \geq k$.
\item $\|a_j\|_{L_p(\mu)} \leq \mu(A_j)^{-1/p'} (k_j-k+1)^{-1}$.
\end{enumerate}
\end{itemize} 
To each $p$-atomic block we attach the quantity
$$
|b|_{\mathrm{H}_{\Sigma}^{1}(\mu)} = \sum_j |\lambda_j|.
$$
Finally, one has the following alternative expression for the norm in $\mathrm{H}_{\Sigma}^{1}(\mu)$:
$$
\|f\|_{\mathrm{H}_{\Sigma}^{1}(\mu)} \sim_p \inf_{\begin{subarray}{c} f=\sum_i b_i = \sum_{i,j} \lambda_{ij} a_{ij} \\ b_i \; p-\mathrm{atomic} \; \mathrm{blocks} \end{subarray}} \sum_{i,j} |\lambda_{ij}|.
$$
The above expression is very similar to the one for the norm of the predual of $\mathrm{RBMO}(\mu)$, although ours yields a greater quantity. This is proven by the inclusion $\mathrm{RBMO}(\mu) \subset \mathrm{RBMO}_\Sigma(\mu)$ and duality. However, it is also easy to find a direct proof, whose details are left to the interested reader. Finally, interpolation holds for $\mathrm{RBMO}_\Sigma(\mu)$ because it holds for any martingale BMO space, see \cite{garsia1973}. This concludes the proof of theorem \ref{RBMO.theoremB}.
\

\

\
A consequence of the inclusion $\mathrm{RBMO}(\mu) \subset \mathrm{RBMO}_\Sigma(\mu)$ is  that $\mathrm{RBMO}(\mu)$ interpolates as a function space. This is something that was not achieved in \cite{tolsaRBMO}, where only interpolation of operators was considered.

\begin{corollary}\label{RBMO.interpolationspaces}
$\mathrm{RBMO}_\Sigma(\mu)$ serves as an interpolation endpoint with respect to the $L_p(\mu)$ scale. That is, we have
$$
\left[\mathrm{RBMO}(\mu),L_1(\mu)\right]_{1/p} \simeq L_p(\mu),
$$
with equivalent norms.
\end{corollary}

\begin{proof}

Indeed, since martingale BMO spaces serve as interpolation endpoints, we have
\begin{eqnarray*}
L_p(\mu) & = & \left[L_{\infty}(\mu),L_1(\mu)\right]_{1/p} \\ 
& \subseteq & \left[\mathrm{RBMO}(\mu),L_1(\mu)\right]_{1/p} \\
& \subseteq &  \left[\mathrm{RBMO}_{\Sigma}(\mu),L_1(\mu)\right]_{1/p} \simeq L_p(\mu).
\end{eqnarray*}
\end{proof}

\begin{remark} \label{RBMO.metricspaces}
Theorem \ref{RBMO.theoremA} holds true in the setting of geometrically doubling metric spaces. Therefore, all our results up to now generalize in a natural way to the context of geometrically doubling metric spaces equipped with an upper doubling measure (see \cite{hytonenRBMO} for details).
\end{remark}


\section{\bf Calder\'on-Zygmund operators} \label{RBMO.section3}

We now consider applications to Calder\'on-Zygmund theory of our results in the previous sections. For us, a ($n$-dimensional) Calder\'on-Zygmund operator $T$ will be a linear operator which is bounded on $L_2(\mu)$ with an associated kernel $k$ for which the representation
$$
Tf(x) = \int k(x,y) f(y) d\mu(y)
$$
holds for $x$ away from the support of sufficiently nice $f$. Additionally, the kernel $k$ is assumed to satisfy the standard size condition and Lipschitz smoothness:
\begin{equation*}
|k(x,y)| \lesssim \frac{1}{|x-y|^n} \;\mathrm{when}\; \; x\not=y, 
\end{equation*}
\begin{equation*}
|k(x,y)-k(x',y)| + |k(y,x)-k(y,x')| \lesssim \frac{1}{|x-y|^n}\frac{|x-x'|^\gamma}{|x-y|^{\gamma}} \;\mathrm{when}\; |x-x'| \leq \frac12 |x-y|.
\end{equation*}
Under these conditions, we know from \cite{tolsaRBMO} that $T$ maps $L_\infty(\mu)$ into $\mathrm{RBMO(\mu)}$ and therefore $T:L_\infty(\mu) \to \mathrm{RBMO}_\Sigma(\mu)$, by the inclusion $\mathrm{RBMO}(\mu) \subset \mathrm{RBMO}_\Sigma(\mu)$. A direct proof is also possible, but we will postpone it until the next section, in which we will consider a more general setting that includes this |and, in fact, one in which Tolsa's RBMO space is not defined|.
\

\

\
We shall now focus on the proof of theorem \ref{RBMO.theoremC}. Our starting point is a result of H\"anninen (see \cite{hanninen2015}) which is a generalization of the so-called Lerner's formula in \cite{lerner2010}. To state it, we need to introduce some terminology: a family of measurable sets $\mathcal{S}$ is called $\eta$-sparse (or only sparse if $\eta=1/2$) if for each $A\in \mathcal{S}$, there exists a measurable $E_A \subset A$ with two properties:
\begin{itemize}
\item $\mu(E_A) \geq \eta \mu(A)$.
\item For each pair $A,B \in \mathcal{S}$ with $A\not=B$, $E_A \cap E_B = \emptyset$.
\end{itemize}
The $\lambda$-oscillation of a function $f$ on a set $A$, denoted $\omega_\lambda(f;A)$, is defined as
$$
\omega_\lambda(f;A) = \inf_{\begin{subarray}{c} A' \subset A \\ \mu(A') \geq \lambda \mu(A) \end{subarray}} \sup_{x,y \in A'} \left|f(x)-f(y) \right| .
$$
Finally, a median of a function $f$ on a set $A$ is a (possibly non unique) number $m_A(f)$ which satisfies
$$
\max\left\{\mu\left(A \cap \{f > m_A(f)\} \right),\mu\left(A \cap \{f < m_A(f)\} \right)\right\}\leq \frac{1}{2} \mu(A).
$$
We can now restate H\"anninen's theorem in the way that we intend to use it: fix some set $Q_0 \in \Pi(\Sigma_k)$ for some $k \in \mathbb{Z}$. Then there exists a sparse family $\mathcal{S}$ of sets in $\Pi(\Sigma)$ such that
\begin{equation} \label{RBMO.hanninen}
|f-m_{Q_0}(f)| \chi_{Q_0} \lesssim_{\lambda} \sum_{Q \in \mathcal{S}} \left(\omega_{\lambda}(f;Q) +\left|m_Q(f)-m_{\widehat{Q}}(f) \right|\right) \chi_Q,
\end{equation}
for any measurable $f$ which is supported on $Q_0$. In \cite{hanninen2015}, \eqref{RBMO.hanninen} is stated for a filtration of $\sigma$-algebras of dyadic cubes on $\mathbb{R}^d$. However, its proof works for any filtration of atomic $\sigma$-algebras, so we can use the result with the filtration $\Sigma$. We can use this version of Lerner's oscillation formula to prove theorem \ref{RBMO.theoremC}.

\begin{proof}[Proof of theorem \ref{RBMO.theoremC}]
Fix a function $f$, and assume qualitatively that it belongs to $L_1(\mu)$ and has compact support contained in $Q_0 \in \Pi(\Sigma)$.  We apply \eqref{RBMO.hanninen} to $Tf$ to obtain
\begin{equation}\label{RBMO.equation32}
|Tf(x)-m_{Q_0}(Tf)| \chi_{Q_0}(x) \lesssim \sum_{Q \in \mathcal{S}} \left(\omega_{\lambda}(Tf;Q) +\left|m_Q(Tf)-m_{\widehat{Q}}(Tf) \right|\right) \chi_Q(x).
\end{equation}
We now estimate each of the terms in \eqref{RBMO.equation32} separately. On the one hand, we may decompose $Tf = T(f\chi_{56B_Q}) + T(\chi_{(56B_Q)^c})=: Tf_1 + Tf_2$. Then, by the weak $(1,1)$ boundedness of $T$ (\cite{ntvWeakType,tolsaWeakType}), we get that 
$$
\mu \left(\left\{x\in Q: |T(f_1)|>  C_1 \mu(Q)^{-1} \int |f_1|d\mu \right\} \right) \leq \|T\|_{L_1(\mu) \to L_{1,\infty}(\mu)} C_1^{-1} \mu(Q). 
$$
Therefore, choosing $C_1=C_1(\mu,\|T\|_{L_1(\mu) \to L_{1,\infty}(\mu)} )$ appropriately, we get 
$$
\mu\left(\left\{x\in Q: |T(f_1)|> C_1 \mu(Q)^{-1} \int |f_1|d\mu\right\}\right) \leq \frac{1}{4} \mu(Q)
$$
(say) and therefore
$$
\omega_{\lambda} (Tf_1;Q) \lesssim \frac{1}{\mu(\alpha B_Q)} \int_{56B_Q}|f|\; d\mu
$$
(recall that $B_Q$ is doubling). On the other hand, if $x,y \in Q$, we have
\begin{eqnarray*}
|Tf_2(x) -Tf_2(y)| & \leq & \sum_{j=0}^{\infty} \int_{56B_{Q^{(j+1)}}\setminus 56B_{Q^{(j)}}} |k(x,z) - k(y,z)| |f(z)| d\mu(z) \\
& \lesssim & \sum_{j=0}^{\infty} A^{-\gamma j} \frac{1}{r(B_{Q^{(j)}})^n} \int_{56B_{Q^{(j+1)}}} |f(z)| d\mu(z) \\
& \lesssim & \sup_{j} \frac{1}{\mu(\alpha B_{Q^{(j+1)}})} \int_{56B_{Q^{(j+1)}}} |f| \; d\mu. \\
\end{eqnarray*}
Therefore, we arrive at 
$$
\omega_{\lambda} (Tf,Q) \lesssim \inf_{x\in Q} \tilde{M}_{\mathscr{D}} f(x)
$$
for some $1/4 < \lambda < 1/2$. We are left with the median term, for which we use the notation $m_{Q,x} (f(x,y))$ for the median in the $x$ variable of the two variable function $f(x,y)$. Using the monotonicity and linearity (for constants!) of the median we get
\begin{eqnarray*}
|m_Q(Tf) - m_{\widehat{Q}}(Tf)| & = & |m_{Q,x}(m_{\widehat{Q},y}(Tf(x) -Tf(y)))| \\
& \leq & |m_Q(T(f\chi_{56B_Q}))| +|m_{\widehat{Q}}(T(f\chi_{56B_{\widehat{Q}}}))| + \sup_{x\in Q } |T(f\chi_{56B_{\widehat{Q}} \setminus 56B_{Q}})(x)| \\
& + & \sup_{x,y \in Q} \sum_{j=1}^{\infty} \int_{56B_{\widehat{Q}^{(j+1)}}\setminus 56B_{\widehat{Q}^{(j)}}} |k(x,z) - k(y,z)| |f(z)| d\mu(z) \\
& =: & \mathrm{I} + \mathrm{II} + \mathrm{III} + \mathrm{IV}. \\ 
\end{eqnarray*}
For $\mathrm{I}$ and $\mathrm{II}$ we just use again that $B_Q$ and $B_{\widehat{Q}}$ are doubling and the weak type $(1,1)$ of $T$ to get
$$
\mathrm{I} + \mathrm{II} \lesssim \frac{1}{\mu(\alpha B_Q)} \int_{56B_Q} |f| \; d\mu + \frac{1}{\mu(\alpha B_{\widehat{Q}})} \int_{56B_{\widehat{Q}}} |f| \; d\mu. 
$$
The term $\mathrm{III}$ is estimated arguing as in the proof of the inclusion $\mathrm{RBMO}(\mu) \subset \mathrm{RBMO}_\Sigma(\mu)$. If $x\in Q$, then
\begin{eqnarray*}
|T(f\chi_{56B_{\widehat{Q}} \setminus 56B_{Q}})(x)| & \leq & \sum_{\begin{subarray}{c} j\geq 0 \\ Q^{(j)}\subset \widehat{Q}\end{subarray}} \int_{56B_{Q^{(j+1)}} \setminus 56B_{Q^{(j)}}} |k(x,y)| |f(y)| d\mu(y) \\
& \lesssim & \sum_{\begin{subarray}{c} j\geq 0 \\ Q^{(j)}\subset \widehat{Q}\end{subarray}} \frac{1}{r(B_{Q^{(j)}})^n} \int_{56B_{Q^{(j+1)}}} |f| d\mu \\
& \lesssim & \left( \sup_j \frac{1}{\mu(\alpha B_{Q^{(j+1)}})} \int_{56B_{Q^{(j+1)}}} |f| d\mu \right) \sum_{\begin{subarray}{c} j\geq 0 \\ Q^{(j)}\subset \widehat{Q}\end{subarray}} \frac{\mu(\alpha B_{Q^{(j+1)}})}{r(\alpha B_{Q^{(j+1)}})^n} \\
& \lesssim & \inf_{y \in Q} \tilde{M}_{\mathscr{D}} f(y), \\
\end{eqnarray*}
where we define
$$
\tilde{M}_{\mathscr{D}} f(x) = \sup_{x\in Q \in \mathscr{D}} \frac{1}{\mu(\alpha B_Q)} \int_{56B_Q} |f| \; d\mu.
$$
In the last step above we have used \eqref{RBMO.keyproperty} as in the proof of theorem \ref{RBMO.theoremA}. Finally, the term $\mathrm{IV}$ is smaller than one of the terms that already appeared in the estimate of the oscillation. We also have to deal with the term $m_{Q_0}(Tf)$, but this is done in the usual way: since $f$ is supported on $Q_0$,  we have
$$
m_{Q_0}(Tf) \leq 2\frac{\|Tf\|_{L_{1,\infty}(\mu)}}{\mu(Q_0)} \lesssim \langle |f|\rangle_{Q_0},
$$
so we can add the resulting quantity to the right hand side of the estimate via the triangle inequality. This may change the sparseness constant of the family $\mathcal{S}$ that we have enlarged with the addition of $Q_0$ from $\frac12$ to $\frac14$. The proof of the pointwise estimate is completed. We can therefore shift our attention to the weighted estimate. We say that a nonnegative locally integrable $w$ which is positive a.e. is an $A_2(\mu)$ weight if
$$
[w]_{A_2(\mu)} := \sup_{Q \; (\alpha',\beta')-\mathrm{doubling}} \frac{w(Q)}{\mu(Q)}\frac{w^{-1}(Q)}{\mu(Q)}.
$$
In the formula above, we denote as usual $w(Q)=\int_Q w \; d\mu$ and $w^{-1}(Q) = \int_Q 1/w \; d\mu$, and the pair $\alpha', \beta' > (\alpha')^d$ is fixed. The sets $Q$ may be Euclidean balls or cubes with sides parallel to the axes. It does not matter which of them we use, up to a change in the associated doublingness constants. We have
$$
|Tf(x)| \lesssim \sum_{Q \in \mathcal{S}} \langle \tilde{M}_{\mathscr{D}} f \rangle_{56B_Q} \chi_Q (x) =: \tilde{\mathcal{A}}_{\mathcal{S}} \left(\tilde{M}_{\mathscr{D}} f \right)(x).
$$
To control the maximal term, we start with a pointwise estimate. We are going to use a dyadic covering lemma like, say, the one in \cite{conde2013}. This means that there exist $d+1$ usual dyadic systems denoted by $\tilde{\mathscr{D}}^j$, $0\leq j \leq d$, such that for any ball $B$, we can find a dyadic cube $Q' \in \tilde{\mathscr{D}}^j$ for some $j$ such that 
$$
B \subset Q' \subset c_d B.
$$ 
Assume $\alpha$ has been chosen so big that
$$
56B_Q \subset Q' \subset \frac{\alpha}{2^d} B_Q.
$$ 
This implies that $Q'$ is $(2,\beta)$-doubling. Then, given $x\in Q$ and its associated $Q' \in \tilde{\mathscr{D}}^j$ for some $j$, we can compute
\begin{eqnarray*}
\frac{1}{\mu(\alpha B_Q)}\int_{56B_Q} |f| \; d\mu & = & \left(\frac{w(Q')}{\mu(\alpha B_Q)} \frac{w^{-1}(Q')}{\mu(Q')} \right) \frac{\mu(Q')}{w(Q')} \frac{1}{w^{-1}(Q')} \int_{Q'} (|f|w) \; w^{-1}d\mu \\
& \leq & [w]_{A_2(\mu)} \frac{\mu(Q')}{w(Q')} \inf_{y\in Q'} M_{\tilde{\mathscr{D}}^j}^{w^{-1}d\mu} (|f|w)(y) \\
& \leq & [w]_{A_2(\mu)} \frac{1}{w(Q')} \int_{Q'} M_{\tilde{\mathscr{D}}^j}^{w^{-1}d\mu} (|f|w) \; d\mu \\
& \leq & [w]_{A_2(\mu)} \inf_{y \in Q'} M_{\tilde{\mathscr{D}}^j}^{wd\mu} \left(M_{\tilde{\mathscr{D}}^j}^{w^{-1}d\mu} (|f|w)w^{-1}\right)(y) \\
& \leq & [w]_{A_2(\mu)} M_{\tilde{\mathscr{D}}^j}^{wd\mu} \left(M_{\tilde{\mathscr{D}}^j}^{w^{-1}d\mu} (|f|w)w^{-1}\right)(x),\\
\end{eqnarray*}
where 
$$
M_{\tilde{\mathscr{D}}^j}^{\nu} f(x)= \sup_{x\in Q \in \tilde{\mathscr{D}}^j} \nu(Q)^{-1} \int_Q |f| \; d\nu.
$$
Now, since $M_{\tilde{\mathscr{D}}^j}^{\nu}$ is bounded on $L_2(\nu)$ with norm independent of $\nu$, we get
\begin{eqnarray*}
\|\tilde{M}_{\mathscr{D}} f\|_{L_2(wd\mu)} & \leq & [w]_{A_2(\mu)} \sum_{j=0}^d \left\| M_{\tilde{\mathscr{D}}^j}^{wd\mu} \left(M_{\tilde{\mathscr{D}}^j}^{w^{-1}d\mu} (|f|w)w^{-1}\right) \right\|_{L_2(wd\mu)} \\
& \lesssim & [w]_{A_2(\mu)} \max_{0\leq j \leq d} \left\| M_{\tilde{\mathscr{D}}^j}^{w^{-1}d\mu} (|f|w) \right\|_{L_2(w^{-1}d\mu)} \\
& \lesssim & [w]_{A_2(\mu)} \| f\|_{L_2(wd\mu)}. \\
\end{eqnarray*}

\noindent
On the other hand, all $Q\in \mathcal{S}$ are doubling, which means that each ball $56B_Q$ is contained in a dyadic cube $Q'\in \tilde{\mathscr{D}}^j$ for some $j \in \{0,\ldots,d\}$ which is $(\alpha',\beta)$-doubling. The family of associated cubes $\{Q'\}_{Q \in \mathcal{S}}$ is then $\eta$-sparse, for some $\eta$ that only depends on $\alpha'$ and $d$. Then, we may apply the argument in \cite{cruz-uribe2012} (to $d+1$ $\eta$-sparse dyadic operators) to conclude
$$
\|\mathcal{A}_{\mathcal{S}}\|_{L_2(wd\mu)\to L_2(wd\mu)} \lesssim [w]_{A_2(\mu)},
$$
which yields the desired result.
\end{proof}

\begin{remark}
It may be that the bound that we have found for the norm of $T$ in terms of $[w]_{A_2(\mu)}$ is not the best possible. Definitely, it is not for the Lebesgue measure, and in that case our method is not optimal. In any case one may wonder, in view of the results in \cite{tolsaWeighted}, if the quantity $[w]_{A_2}$ is the most natural one to study the dependence on the weight in the nondoubling setting.
\end{remark}

We end this section with a discussion on Calder\'on-Zygmund decompositions for general measures. The classical decomposition is a tool that allows to decompose an integrable function into pieces that are well suited for Calder\'on-Zygmund estimates, yielding the weak type $(1,1)$ inequality for $L_2$-bounded Calder\'on-Zygmund operators. The same decomposition can be used to study the weak type $(1,1)$ boundedness of dyadic models like Haar shifts. However, when the underlying measure is nondoubling, as in the setting of this paper, singular integrals and their dyadic models require different tools adapted to their particular structure. On the one hand, since centered dilations of balls/cubes are important for Calder\'on-Zygmund operators |to ensure that one goes far away from the diagonal in kernel estimates| Tolsa's decomposition (see \cite{tolsaWeakType}) has a centered nature, and dyadic systems do not play any role in it. On the other hand, the problem with the study of Haar shifts associated to a dyadic lattice $\tilde{\mathscr{D}}$ |when the complexity is nonzero; otherwise, the problem was understood long ago via Gundy's martingale decomposition| is related to the interaction between cubes an their dyadic ancestors. The `right' decomposition in this case, found in \cite{lopezsanchez-martell-parcet2014}, reads as follows: given $f \in L_1(\mu)$ and a height $\lambda >0$, if $\{Q_j\}_j$ is the family of maximal dyadic cubes with respect to the property $\langle|f|\rangle_Q>\lambda$, then $f = g + b + \beta$, where 
$$
b = \sum_j b_j = \sum_j(f-\langle f \rangle_{Q_j}) \chi_{Q_j} \; \mathrm{and} \; \sum_j \|b_j\|_{L_1(\mu)} \lesssim \|f\|_{L_1(\mu)},
$$
$$
\beta = \sum_j \beta_j = \sum_j (\langle f \rangle_{Q_j} - \langle f \rangle_{Q_j^{(1)}}) \left(\chi_{Q_j} - \frac{\mu(Q_j)}{\mu(Q_j^{(1)})} \chi_{Q_j^{(1)}}\right); \;\mathrm{and}\;  \sum_j \|\beta_j\|_{L_1(\mu)} \lesssim \|f\|_{L_1(\mu)},
$$
$$
g = f - b - \beta, \; \|g\|_{L_2(\mu)} \lesssim \lambda \|f\|_1.
$$
Notice that the maximal cubes $Q_j$ need not be doubling. Interestingly, it turns out that the natural condition on the measure $\mu$ in this case is not the polynomial growth condition, but \emph{equilibration}. In the one-dimensional case, this means the following: if the brother of a cube $Q$ is $Q_{\mathrm{b}}$, define
$$
m(Q) = \frac{\mu(Q) \mu(Q_{\mathrm{b}})}{\mu(Q^{(1)})^2}.
$$
Then a locally finite Borel measure $\mu$ is equilibrated if $m(Q) \sim m(Q^{(1)})$ for all $Q \in \tilde{\mathscr{D}}$, a condition that is neither implied nor implies linear growth. As a consequence of this, Haar shift theory and Calder\'on-Zygmund theory are different when $\mu$ is nondoubling. However, if we substitute the dyadic system $\tilde{\mathscr{D}}$ by the filtration $\Sigma$, we obtain a simple unified approach to both theories/decompositions:

\begin{theorem}\label{RBMO.CZ}
Let $0\leq f\in L_1(\mu)$ and $\lambda>0$ (or $\lambda > \|f\|_{L_1(\mu)} / \|\mu\|$ if $\mu$ is finite). Consider the family of maximal sets of $\Pi(\Sigma)$ with respect to the property $\langle f \rangle_Q > \lambda$. We can write $f=g+\tilde{\beta}$, where
$$
\tilde{\beta} = \sum_{k\in \mathbb{Z}} \varphi_k = \sum_{k\in \mathbb{Z}} \sum_{\begin{subarray}{c} Q \in \Pi(\Sigma_k) \\ Q \; \mathrm{maximal} \end{subarray}} \left(f\chi_{Q} - \mathsf{E}_{\Sigma_{k-1}}\left[f\chi_{Q}\right] \right), \; g= f-\tilde{\beta},
$$
and we have $\sum_j \|\varphi_j\|_{L_1(\mu)} \lesssim \|f\|_{L_1(\mu)}$ and $\|g\|_{L_2(\mu)} \lesssim \lambda \|f\|_{L_1(\mu)}$.
\end{theorem}

The proof of theorem \ref{RBMO.CZ} is essentially the same than that of theorem A in \cite{lopezsanchez-martell-parcet2014}. However, the additional properties satisfied by the filtration $\Sigma$ yield
$$
T: L_1(\mu) \to L_{1,\infty}(\mu)
$$
for Calder\'on-Zygmund operators $T$, hence recovering Tolsa's result \cite{tolsaWeakType}. Also, if $\mu$ is equilibrated with respect to $\Sigma$, we can establish the same weak $L_1$ estimates for Haar shift operators |defined with respect to $\Sigma$|. The details are left to the reader.


\section{\bf Noncommutative $\mathrm{RBMO}$}
\label{RBMO.section4}

We now study the validity of our results of section \ref{RBMO.sectionRBMO} in the operator valued setting. Our models are the operator valued $\mathrm{BMO}$ in \cite{meiMemoirs} and the large $\mathrm{BMO}$ spaces in \cite{condealonso-mei-parcet}. Hence, assume that $\mathcal{M}$ is a |possibly noncommutative| von Neumann algebra equipped with a trace $\tau$. This is sometimes called a `noncommutative measure space', see \cite{pisier-xu} for a nice introduction of the subject. We consider the semicommutative von Neumann algebra $\mathcal{A}$ which is the von Neumann algebra tensor product of $(\mathcal{M},\tau)$ with $L_\infty(\mu)$, that is 
$$
(\mathcal{A},\varphi) = L_\infty(\mu) \overline{\otimes} (\mathcal{M},\tau).
$$
(see \cite{parcet2009} for details on a similar setting). This is both a von Neumann algebra and a vector valued $L_\infty$ space. The trace on $\mathcal{A}$ |the noncommutative integral| acts on functions $f:\mathbb{R}^d \to \mathcal{M}$ by
$$
\varphi(f) = \int d\mu \otimes \tau (f) = \int_{\mathbb{R}^d} \tau\left[f(x) \right] d\mu(x).
$$
The noncommutative $L_p$ scale associated to $\mathcal{A}$ is the family of spaces
$$
L_p(\mathcal{A}) := L_p(\mathbb{R}^d,\mu;L_p(\mathcal{M})), \; p\geq 1.
$$
Let us introduce our $\mathrm{RBMO}_\Sigma$ norm in this setting. Define the absolute value (for operators) by $|x|^2 = x^*x$. Notice that in the noncommutative setting, $|x| \not= |x^*|$. Given a $\tau$-measurable function $f:\mathbb{R}^d \to L_0(\mathcal{M})$ its norm in $\mathrm{RBMO}_\Sigma(\mathcal{A})$ is its (noncommutative) martingale $\mathrm{BMO}$ norm with respect to the filtration $\Sigma$:
$$
\|f\|_{\mathrm{RBMO}_\Sigma^c} = \sup_{k \in \mathbb{Z}} \left\| \mathsf{E}_{\Sigma_k} \left|f-\mathsf{E}_{\Sigma_{k-1}}f \right|^2 \right\|_{\mathcal{A}}^{\frac12}.
$$
We say that $f$ belongs to $\mathrm{RBMO}_\Sigma = \mathrm{RBMO}_\Sigma(\mu)$ if $\max\{\|f\|_{\mathrm{RBMO}_\Sigma^c},\|f^*\|_{\mathrm{RBMO}_\Sigma^c} \}<\infty$. Notice that our conditional expectations 
$$
\mathsf{E}_{\Sigma_k} : \mathcal{A} \to L_\infty(\mathbb{R}^d,\Sigma_k,\mu;\mathcal{M})
$$
are again operator valued: the average of a function $f$ over an atom in $\Pi(\Sigma)$ is still an element of $\mathcal{M}$. In any case, $\mathrm{RBMO}_\Sigma(\mathcal{A})$ is a noncommutative martingale $\mathrm{BMO}$ norm, and therefore we automatically have that the following nice properties hold:
\begin{itemize}
\item \textbf{John-Nirenberg inequality:} The noncommutative John-Nirenberg inequality was established in \cite{junge-musat,hong-meiJN}, and by the semicommutative nature of $\mathcal{A}$ it implies that
$$
\|f\|_{\mathrm{RBMO}_\Sigma^c(\mathcal{A})} \sim_p \sup_{k\in \mathbb{Z}}  \left\|\mathsf{E}_{\Sigma_k}\left|f - \mathsf{E}_{\Sigma_{k-1}}f \right|^p\right\|_{\mathcal{A}}^\frac1p.
$$
\item \textbf{Interpolation:} The noncommutative interpolation theorem with $\mathrm{BMO}$ (see \cite{musat}) gives the interpolation of $\mathrm{RBMO}_\Sigma(\mathcal{A})$:
$$
\left[\mathrm{RBMO}_\Sigma(\mathcal{A}), L_1(\mathcal{A}) \right]_{1/p} = L_p(\mathcal{A}).
$$
\item \textbf{Fefferman-Stein duality:} the predual of $\mathrm{RBMO}_\Sigma(\mathcal{A})$ is, as in the commutative case, the Hardy space $\mathrm{H}_1$ associated to the noncommutative dyadic square function. However, as shown in \cite{condealonso-parcet}, it also admits an atomic block decomposition very similar to that explained in section \ref{RBMO.sectionRBMO}.
\end{itemize}

To finish our study, let us consider Calder\'on-Zygmund $L_\infty-\mathrm{RBMO}$ estimates in the noncommutative setting. We start defining what we understand by a Calder\'on-Zygmund operator. We essentially follow \cite{condealonso-mei-parcet,junge-mei-parcetSmooth}. Consider kernels $k: (\mathbb{R}^d \times \mathbb{R}^d) \setminus \Delta \to \mathcal{L}(L_0(\mathcal{M}))$ defined away from the diagonal $\Delta $ of $\mathbb{R}^d \times \mathbb{R}^d$ and which take values in linear maps on $\tau$-measurable operators. The standard H\"ormander kernel condition takes the same form in this setting when we replace the absolute value by the norm in the algebra $\mathcal{B}(\mathcal{M})$ of bounded linear operators acting on $\mathcal{M}$:
\begin{equation}\label{RBMO.hormander}
\sup_{\begin{subarray}{c} \mathrm{B} \, \mathrm{ball} \\ z_1, z_2 \in \mathrm{B} \end{subarray}} \, \int_{\mathbb{R}^d \setminus \alpha \mathrm{B}}  \big\| k(z_1,x) - k(z_2,x) \big\|_{\mathcal{B}(\mathcal{M})} + \big\| k(x,z_1) - k(x,z_2) \big\|_{\mathcal{B}(\mathcal{M})} \, d\mu(x) \, < \, \infty.
\end{equation}
Define a Calder\'on-Zygmund operator in $(\mathcal{A}, \varphi)$ as any linear map $T$ satisfying the following properties:
\begin{itemize}
\item $T$ is bounded on $L_\infty(\mathcal{M}; L_2^r(\mu))$:
$$
\hskip10pt \Big\| \int_{\mathbb{R}^d} |Tf(x)^*|^2 \, d\mu(x) \Big\|_{\mathcal{M}}^{\frac12} \, \lesssim \, \Big\| \int_{\mathbb{R}^d} |f(x)^*|^2 \, d\mu(x) \Big\|_{\mathcal{M}}^{\frac12}.
$$
\item $T$ is bounded on $L_\infty(\mathcal{M}; L_2^c(\mu))$: 
$$
\hskip10pt \Big\| \int_{\mathbb{R}^d} |Tf(x)|^2 \, d\mu(x) \Big\|_{\mathcal{M}}^{\frac12} \, \lesssim \, \Big\| \int_{\mathbb{R}^d} |f(x)|^2 \, d\mu(x) \Big\|_{\mathcal{M}}^{\frac12}.
$$
\item The kernel representation 
$$
Tf(x) \, = \, \int_{\mathbb{R}^d} k(x,y) (f(y)) \, d\mu(y) \quad \mbox{holds for} \quad x \notin \mathrm{supp}_{\mathbb{R}^d} (f)
$$ 
and some kernel $k$ satisfying the operator valued H\"ormander condition \eqref{RBMO.hormander}. Here $\mathrm{supp}_{\mathbb{R}^d}(f)$ denotes the support of $f$ as a function on $\mathbb{R}^d$ (and not as an operator in $\mathcal{A}$). 
\item $k$ satisfies the $n$-dimensional size condition
$$
\|k(x,y)\|_{\mathcal{M}} \lesssim \frac{1}{|x-y|^n}, \; x\not= y.
$$
\end{itemize}
The first two conditions are the natural replacement for the standard $L_2$ boundedness, see \cite{junge-mei-parcetSmooth} for explanations. Let us prove theorem \ref{RBMO.theoremD}:

\begin{proof}[Proof of theorem \ref{RBMO.theoremD}]
We only prove that $T:\mathcal{A} \to \mathrm{RBMO}_{\Sigma}^c(\mathcal{A})$ using the $L_\infty(\mathcal{M}; L_2^c(\mu))$-boundedness. Then, the bound $T:\mathcal{A} \to \mathrm{RBMO}_{\Sigma}^r(\mathcal{A})$ will follow analogously from $L_\infty(\mathcal{M}; L_2^r(\mu))$-boundedness and the proof will be complete. Recall that 
$$
\|f\|_{\mathrm{RBMO}_{\Sigma}^c(\mathcal{A})} = \sup_{k\in\mathbb{Z}} \Big\|\mathsf{E}_{\Sigma_k}|f-\mathsf{E}_{\Sigma_{k-1
}}f|^2 \Big\|_{\mathcal{M}}^{\frac12} = \sup_{Q \in \Pi(\Sigma)} \left\|  \dashint_Q \left|f - \langle f\rangle_{\widehat{Q}} \right|^2 d\mu\right\|_{\mathcal{M}}^{\frac12}.
$$
Fix $Q\in\Pi(\Sigma)$ and $f\in \mathcal{A}$. We can write
\begin{eqnarray*}
\left\|\frac{1}{\mu(Q)} \int_Q \left|Tf - \langle Tf\rangle_{\widehat{Q}}\right|^2 d\mu\right\|_{\mathcal{M}}^{\frac12} & \leq & \left\|\left\langle |T(f\chi_{\alpha B_Q})|^2\right\rangle_Q \right\|_{\mathcal{M}}^{\frac12} + \left\|\left\langle \left|\langle T(f\chi_{\alpha B_{\widehat{Q}}}) \rangle_{\widehat{Q}}\right|^2\right\rangle_Q\right\|_{\mathcal{M}}^{\frac12} \\
& + & \left\|\left\langle\left|\int_{\alpha B_{\widehat{Q}}\setminus \alpha B_Q}k(\cdot,y)f(y) d\mu(y) \right|^2 \right\rangle_Q \right\|_{\mathcal{M}}^{\frac12} \\
& + & \left\| \dashint_Q \left| \dashint_{\widehat{Q}} \int_{\mathbb{R}^d \setminus \alpha B_{\widehat{Q}}} (k(x,y)-k(z,y))f(y)d\mu(y)d\mu(z)\right|^2 d\mu(x)\right\|_{\mathcal{M}}^{\frac12} \\
& =: & \mathrm{I} + \mathrm{II}+ \mathrm{III}+ \mathrm{IV}.
\end{eqnarray*}

\noindent 
By the Kadison-Schwarz inequality for unital completely positive maps \cite{lanceBook}, we know that $|\langle f\rangle_Q|^2 \leq \langle |f|^2\rangle_Q$. This, combined with $L_\infty(\mathcal{M}; L_2^c(\mu))$-boundedness and the fact that $B_{\widehat{Q}}$ is $(\alpha,\beta)$-doubling, yields
\begin{eqnarray*}
\mathrm{II} & \leq & \left\|\left\langle \left| T(f\chi_{\alpha B_{\widehat{Q}}}) \right|^2\right\rangle_{\widehat{Q}}\right\|_{\mathcal{M}}^{\frac12} \\
& \lesssim & \left\|\frac{1}{\mu(B_{\widehat{Q}})} \int_{\mathbb{R}^d} \left| f\chi_{\alpha B_{\widehat{Q}}} \right|^2 d\mu\right\|_{\mathcal{M}}^{\frac12} \\
& \leq & \esssup_{x \in \alpha B_{\widehat{Q}}} \|f(x)\|_{\mathcal{M}} \left(\frac{\mu(\alpha B_{\widehat{Q}})}{\mu(B_{\widehat{Q}})} \right)^{\frac12} \lesssim \|f\|_{\mathcal{A}}. \\
\end{eqnarray*}
A similar computation holds for $\mathrm{I}$. Next, we use the triangle inequality and the $C^*$ property $\|x^*x\|_{\mathcal{M}} = \|x\|_{\mathcal{M}}^2$ to compute
\begin{eqnarray*}
\mathrm{III} & \leq & \left(\left\langle \left( \int_{\alpha B_{\widehat{Q}} \setminus \alpha B_Q} \|k(\cdot,y) f(y)\|_\mathcal{M} d\mu(y) \right)^2 \right\rangle_Q \right)^{\frac12} \\
& \leq & \|f\|_{\mathcal{A}} \sup_{x\in Q} \int_{\alpha B_{\widehat{Q}} \setminus \alpha B_Q} \|k(x,y)\|_{\mathcal{B}(\mathcal{M})} d\mu(y) \\
& \lesssim & \|f\|_{\mathcal{A}} \sup_{x\in Q} \int_{\alpha B_{\widehat{Q}} \setminus \alpha B_Q} \frac{1}{|x-y|^n} d\mu(y) \lesssim \|f\|_{\mathcal{A}}, \\
\end{eqnarray*}
by the size condition of $k$ and (iv) of theorem \ref{RBMO.theoremA}. The last term can be estimated using the operator valued H\"ormander condition:
\begin{eqnarray*}
\mathrm{IV} & \leq & \left( \dashint_Q \left\| \dashint_{\widehat{Q}} \int_{\mathbb{R}^d \setminus \alpha B_{\widehat{Q}}} (k(x,y) - k(z,y))f(y) d\mu(y)  d\mu(z) \right\|_{\mathcal{M}}^2 d\mu(x) \right)^{\frac12} \\
& \leq & \left( \dashint_Q \left( \dashint_{\widehat{Q}} \left\| \int_{\mathbb{R}^d \setminus \alpha B_{\widehat{Q}}} (k(x,y) - k(z,y))f(y) d\mu(y) \right\|_{\mathcal{M}} d\mu(z) \right)^2 d\mu(x) \right)^{\frac12} \\
& \leq & \|f\|_{\mathcal{A}} \left( \dashint_Q \left( \dashint_{\widehat{Q}} \int_{\mathbb{R}^d \setminus \alpha B_{\widehat{Q}}} \|k(x,y) - k(z,y)\|_{\mathcal{B}(\mathcal{M})} d\mu(y) d\mu(z) \right)^2 d\mu(x) \right)^{\frac12}\\
& \lesssim & \|f\|_{\mathcal{A}} \left(\dashint_Q \left( \dashint_{\widehat{Q}} d\mu(z) \right)^2 d\mu(x)\right)^{\frac12} = \|f\|_{\mathcal{A}}.
\end{eqnarray*}
We are done.
\end{proof}

\begin{remark}
The endpoint estimate just proven and the interpolation property of $\mathrm{RBMO}_\Sigma(\mathcal{A})$ allows to deduce boundedness of Calder\'on-Zygmund operators on $L_p(\mathcal{A})$, $1<p<\infty$. Notice that this is a noncommutative $L_p$ scale, and not a vector valued one. In any case, due to the fact that $L_p(\mathcal{M})$ is UMD for $1<p<\infty$ and observing that $L_p(\mathcal{A})$ is also a vector valued space, this boundedness can be deduced from the vector valued theory whenever the kernel $k(x,y)$ is scalar valued. However, as was discussed in \cite{parcet2009}, the use of a noncommutative interpolation scale yields better constants in the inequalities in terms of dependence on $p$.
\end{remark}

\begin{remark}
As the proof shows, Lipschitz smoothness on the kernel is not required to obtain the $\mathcal{A} \to \mathrm{RBMO}_\Sigma(\mathcal{A})$ endpoint estimate, and one can work only with H\"ormander condition in conjunction with the size condition (of degree equal to the dimension of the measure $\mu$).
\end{remark}

\bibliography{BibliographyRBMO}{}

\def\cprime{$'$}
\begin{thebibliography}{10}

\bibitem{Sukochev}
M.~Caspers, D.~Potapov, F.~Sukochev, and D.~Zanin.
\newblock Weak type commutator and lipschitz estimates: resolution of the
  $\mathrm{N}$azarov-$\mathrm{P}$eller conjecture.
\newblock Preprint arXiv:1506.00778 [math.FA].

\bibitem{conde2013}
J.~M. Conde.
\newblock A note on dyadic coverings and nondoubling {C}alder\'on-{Z}ygmund
  theory.
\newblock {\em J. Math. Anal. Appl.}, 397(2):785--790, 2013.

\bibitem{condealonso-mei-parcet}
J.~M. Conde-Alonso, T.~Mei, and J.~Parcet.
\newblock Large {BMO} spaces vs interpolation.
\newblock {\em Anal. PDE}, 8(3):713--746, 2015.

\bibitem{condealonso-parcet}
J.~M. Conde-Alonso and J.~Parcet.
\newblock Atomic blocks for noncommutative martingales.
\newblock Preprint arXiv:1409.4351 [math.FA], to appear in Indiana Math. J.

\bibitem{condealonso-rey2014}
J.~M. Conde-Alonso and G.~Rey.
\newblock On a pointwise estimate for positive dyadic shifts and some
  applications.
\newblock Preprint arXiv:1409.4351 [math.FA], accepted in Math. Ann.

\bibitem{cruz-uribe2012}
D.~Cruz-Uribe, J.~M. Martell, and C.~P{\'e}rez.
\newblock Sharp weighted estimates for classical operators.
\newblock {\em Adv. Math.}, 229(1):408--441, 2012.

\bibitem{david-mattila}
G.~David and P.~Mattila.
\newblock Removable sets for {L}ipschitz harmonic functions in the plane.
\newblock {\em Rev. Mat. Iberoamericana}, 16(1):137--215, 2000.

\bibitem{garnett-jones}
J.~B. Garnett and P.~W. Jones.
\newblock $\mathrm{BMO}$ from dyadic $\mathrm{BMO}$.
\newblock {\em Pacific J. Math.}, 99(2):351--371, 1982.

\bibitem{garsia1973}
A.~M. Garsia.
\newblock {\em Martingale inequalities: {S}eminar notes on recent progress}.
\newblock W. A. Benjamin, Inc., Reading, Mass.-London-Amsterdam, 1973.
\newblock Mathematics Lecture Notes Series.

\bibitem{hanninen2015}
T.~S. H\"anninen.
\newblock Remark on dyadic pointwise domination and median oscillation
  decomposition.
\newblock Preprint arXiv:1502.05942 [math.FA].

\bibitem{hong-meiJN}
G.~Hong and T.~Mei.
\newblock John-{N}irenberg inequality and atomic decomposition for
  noncommutative martingales.
\newblock {\em J. Funct. Anal.}, 263(4):1064--1097, 2012.

\bibitem{hytonenRBMO}
T.~P. Hyt{\"o}nen.
\newblock A framework for non-homogeneous analysis on metric spaces, and the
  {RBMO} space of {T}olsa.
\newblock {\em Publ. Mat.}, 54(2):485--504, 2010.

\bibitem{hytonenA2}
T.~P. Hyt{\"o}nen.
\newblock The sharp weighted bound for general {C}alder\'on-{Z}ygmund
  operators.
\newblock {\em Ann. of Math. (2)}, 175(3):1473--1506, 2012.

\bibitem{junge-mei-parcetSmooth}
M.~Junge, T.~Mei, and J.~Parcet.
\newblock Smooth {F}ourier multipliers on group von {N}eumann algebras.
\newblock {\em Geom. Funct. Anal.}, 24(6):1913--1980, 2014.

\bibitem{junge-musat}
M.~Junge and M.~Musat.
\newblock A noncommutative version of the {J}ohn-{N}irenberg theorem.
\newblock {\em Trans. Amer. Math. Soc.}, 359(1):115--142, 2007.

\bibitem{lacey2015}
M.~T. Lacey.
\newblock An elementary proof of the $\mathrm{A}_2$ bound.
\newblock Preprint arXiv:1501.05818 [math.FA].

\bibitem{lanceBook}
E.~C. Lance.
\newblock {\em Hilbert {$C^*$}-modules}, volume 210 of {\em London Mathematical
  Society Lecture Note Series}.
\newblock Cambridge University Press, Cambridge, 1995.
\newblock A toolkit for operator algebraists.

\bibitem{lerner2016}
A.~K. Lerner.
\newblock On pointwise estimates involving sparse operators.
\newblock Preprint arXiv:1512.07247 [math.FA].

\bibitem{lerner2010}
A.~K. Lerner.
\newblock A pointwise estimate for the local sharp maximal function with
  applications to singular integrals.
\newblock {\em Bull. Lond. Math. Soc.}, 42(5):843--856, 2010.

\bibitem{lernerPointwise}
A.~K. Lerner.
\newblock On an estimate of {C}alder\'on-{Z}ygmund operators by dyadic positive
  operators.
\newblock {\em J. Anal. Math.}, 121:141--161, 2013.

\bibitem{lerner-nazarov2014}
A.~K. Lerner and F.~Nazarov.
\newblock Intuitive dyadic calculus: the basics.
\newblock Preprint arXiv:1508.05639.

\bibitem{lopezsanchez-martell-parcet2014}
L.~D. L\'{o}pez-S\'{a}nchez, J.~M. Martell, and J.~Parcet.
\newblock Dyadic harmonic analysis beyond doubling measures.
\newblock {\em Adv. Math.}, 267:44--93, 2014.

\bibitem{meiMemoirs}
T.~Mei.
\newblock Operator valued {H}ardy spaces.
\newblock {\em Mem. Amer. Math. Soc.}, 188(881):vi+64, 2007.

\bibitem{musat}
M.~Musat.
\newblock Interpolation between non-commutative {BMO} and non-commutative
  {$L_p$}-spaces.
\newblock {\em J. Funct. Anal.}, 202(1):195--225, 2003.

\bibitem{ntvWeakType}
F.~Nazarov, S.~Treil, and A.~Volberg.
\newblock Weak type estimates and {C}otlar inequalities for
  {C}alder\'on-{Z}ygmund operators on nonhomogeneous spaces.
\newblock {\em Internat. Math. Res. Notices}, (9):463--487, 1998.

\bibitem{ntvTb}
F.~Nazarov, S.~Treil, and A.~Volberg.
\newblock The {$Tb$}-theorem on non-homogeneous spaces.
\newblock {\em Acta Math.}, 190(2):151--239, 2003.

\bibitem{parcet2009}
J.~Parcet.
\newblock Pseudo-localization of singular integrals and noncommutative
  {C}alder\'on-{Z}ygmund theory.
\newblock {\em J. Funct. Anal.}, 256(2):509--593, 2009.

\bibitem{petermichlPhD}
S.~Petermichl.
\newblock {\em Of some sharp estimates involving {H}ilbert transform}.
\newblock ProQuest LLC, Ann Arbor, MI, 2000.
\newblock Thesis (Ph.D.)--Michigan State University.

\bibitem{pisier-xu}
G.~Pisier and Q.~Xu.
\newblock Non-commutative martingale inequalities.
\newblock {\em Comm. Math. Phys.}, 189(3):667--698, 1997.

\bibitem{tolsaRBMO}
X.~Tolsa.
\newblock B{MO}, {$\mathrm{H}^1$}, and {C}alder\'on-{Z}ygmund operators for non
  doubling measures.
\newblock {\em Math. Ann.}, 319(1):89--149, 2001.

\bibitem{tolsaWeakType}
X.~Tolsa.
\newblock A proof of the weak {$(1,1)$} inequality for singular integrals with
  non doubling measures based on a {C}alder\'on-{Z}ygmund decomposition.
\newblock {\em Publ. Mat.}, 45(1):163--174, 2001.

\bibitem{tolsaWeighted}
X.~Tolsa.
\newblock Weighted norm inequalities for {C}alder\'on-{Z}ygmund operators
  without doubling conditions.
\newblock {\em Publ. Mat.}, 51(2):397--456, 2007.

\bibitem{eiderman-volbergSurvey}
A.~L. Volberg and V.~Y. Eiderman.
\newblock Nonhomogeneous harmonic analysis: 16 years of development.
\newblock {\em Uspekhi Mat. Nauk}, 68(6(414)):3--58, 2013.

\end{thebibliography}
\bibliographystyle{abbrv}

\vskip20pt

\hfill \noindent \textbf{Jos\'e M. Conde-Alonso} \\
\null \hfill Departament de Matem\`atiques \\ 
\null \hfill Universitat Aut\`onoma de Barcelona \\ 
\null \hfill 08193, Bellaterra, Barcelona. Spain \\ 
\null \hfill\texttt{jconde@mat.uab.cat}

\vskip5pt


\hfill \noindent \textbf{Javier Parcet} \\
\null \hfill Instituto de Ciencias Matem{\'a}ticas \\ \null \hfill
CSIC-UAM-UC3M-UCM \\ \null \hfill Consejo Superior de
Investigaciones Cient{\'\i}ficas \\ \null \hfill C/ Nicol\'as Cabrera 13-15.
28049, Madrid. Spain \\ \null \hfill\texttt{javier.parcet@icmat.es}

\end{document}